\documentclass[11pt]{article}
\usepackage{fullpage}

\usepackage{latexsym,amssymb,amsmath,graphics,graphicx,float,psfrag, epsfig, color}
\usepackage{bm}
\newcommand{\matrixgeq}{\succeq}
\newcommand{\matrixleq}{\preceq}
\newcommand{\matrixgt}{\succ}

\newcommand{\A}{\mathcal{A}}
\newcommand{\M}{\mathbf{M}}
\newcommand{\bfA}{\mathbf{A}}
\newcommand{\X}{\mathbf{X}}
\newcommand{\x}{\mathbf{x}}
\newcommand{\Y}{\mathbf{Y}}
\newcommand{\V}{\mathbf{V}}
\newcommand{\bfb}{\mathbf{b}}
\newcommand{\R}{\mathbb{R}}
\newcommand{\Rsymn}{\mathcal{S}^n}
\newcommand{\Rsym}{\mathcal{S}}

\newcommand{\z}{\mathbf{z}}
\newcommand{\Q}{\mathbf{Q}}

\newcommand{\e}{\mathbf{e}}

\newcommand{\y}{\mathbf{y}}
\newcommand{\q}{\mathbf{q}}
\newtheorem{theorem}{Theorem}
\newenvironment{proof}[1][Proof]{\begin{trivlist}
\item[\hskip \labelsep {\bfseries #1}]}{\end{trivlist}}
\newtheorem{example}{Example}

\newtheorem{lemma}{Lemma}
\newtheorem{corollary}{Corollary}
\newcommand{\ind}{I}
\newcommand{\I}{\mathbf{I}}
\newcommand{\indpos}{\ind_{\X\matrixgeq0}}
\newcommand{\indaxb}{\ind_{\A(\X)=\bfb}}
\newcommand{\indposaxb}{\ind_{\X\matrixgeq0, \A(\X)=\bfb}}
\newcommand{\blambda}{\bm \lambda}

\newcommand{\brhot}{\tilde {\bm \rho}}

\newcommand{\lam}{\mathbf{\Lambda}}
\newcommand{\lamy}{\lam_\Y}
\newcommand{\lamz}{\lam_\Z}
\newcommand{\lamtilde}{\tilde{\lam}}

\newcommand{\lamxop}{\lam_{\x_0^\perp}}
\newcommand{\ivp}{\I_{V^\perp}}
\newcommand{\proj}[1]{\mathcal{P}_{#1}}
\newcommand{\projvp}{\proj{V^\perp}}
\newcommand{\projxop}{\proj{\x_0^\perp}}
\newcommand{\Ak}{\bfA^{(k)}}

\newcommand{\Qk}{\Q^{(k)}}
\newcommand{\Qinf}{\Q^{(\infty)}}
\newcommand{\lik}{\lambda_i^{(k)}}
\newcommand{\liinf}{\lambda_i^{(\infty)}}
\newcommand{\eps}{\varepsilon}
\newcommand{\Z}{\mathbf{Z}}
\newcommand{\xoo}{\x_{0, \Omega}}
\def \zi {\mathbf{z}_i}

\DeclareMathOperator*{\sgn}{sgn}
\DeclareMathOperator*{\trace}{tr}
\DeclareMathOperator*{\Span}{span}
\DeclareMathOperator*{\range}{range}
\def \<{\langle}
\def \>{\rangle}

\begin{document}
\title{Conditions for Existence of Dual Certificates \\in Rank-One Semidefinite Problems}
\author{Paul Hand \\ $\;$ \\ Massachusetts Institute of Technology, Department of Mathematics, \\ 77 Massachusetts Avenue, Cambridge, MA 02139}
\date{November 2013, Revised February 2014}
\maketitle

\begin{abstract}
Several signal recovery tasks can be relaxed into semidefinite programs with rank-one minimizers.  A common technique for proving these programs succeed is to construct a dual certificate.  Unfortunately, dual certificates may not exist under some formulations of semidefinite programs.  In order to put problems into a form where dual certificate arguments are possible, it is important to develop conditions under which the certificates exist.  In this paper, we provide an example where dual certificates do not exist.  We then present a completeness condition under which they are guaranteed to exist.  For programs that do not satisfy the completeness condition, we present a completion process which produces an equivalent program that does satisfy the condition.  
The important message of this paper is that dual certificates may not exist for semidefinite programs that involve orthogonal measurements with respect to positive-semidefinite matrices.  Such measurements can interact with the positive-semidefinite constraint in a way that implies additional linear measurements.  If these additional measurements are not included in the problem formulation, then dual certificates may fail to exist.
As an illustration, we present a semidefinite relaxation for the task of finding the sparsest element in a subspace.  One formulation of this program does not admit dual certificates.  The completion process  produces an equivalent formulation which does admit dual certificates.  
\end{abstract}

\section{Introduction}

We consider the problem of showing that the rank-one matrix $\X_0 = \x_0 \x_0^t$ minimizes  the semidefinite program 
\begin{align}
\min f(\X) \text{ subject to } \X \matrixgeq 0, \A(\X) = \bfb,  \label{convex-problem}
\end{align}
where $\X \in \Rsymn$ is a symmetric and real-valued $n \times n$ matrix, $f$ is convex and continuous, $\A$ is linear, and $\A(\X_0) = \bfb \in \R^m$.  Let $\langle \X, \Y \rangle = \trace(\Y^t \X)$ be the Hilbert-Schmidt inner product.  Matrix orthogonality is understood to be with respect to this inner product.  The linear measurements $\A(\X) = \bfb$ can be written as
$$\A(\X)_i = \langle \X, \bfA_i \rangle = b_i  \text{ for } i = 1, \ldots, m,$$
for certain symmetric matrices $\bfA_i$.  Note that the adjoint of $\A$ is $\A^*\blambda = \sum_{i=1}^m \lambda_i \bfA_i$.

A common approach for proving that $\X_0$ minimizes \eqref{convex-problem} is to exhibit what is known as a dual certificate at $\X_0$.  Unfortunately, such dual certificates do not always exist.  In this paper, we provide a simple example for which they do not.  We then present a sufficient and weakly necessary condition for when they do exist.  If a dual certificate does not exist, we provide a process for completing the problem into another problem of form \eqref{convex-problem} for which a dual certificate exists.

Understanding the existence of dual certificates is important for two reasons.  First, many papers directly construct them, or approximations thereof, in order to show that a particular program succeeds at finding a desired vector \cite{G2011, CSV2011, ARM2012, LV2012}.  If no dual certificate exists, then the pursuit of such dual certificates in some problems may be futile.  Second, dual certificate existence is important for negative results.  In order to show that $\X_0$ is not the solution to  \eqref{convex-problem}, one approach is to show that a dual certificate at $\X_0$ does not exist.  See \cite{LV2012} for an example.  For this proof method to work, it must be established that a dual certificate would exist if $\X_0$ were a minimizer.

\subsection{Motivating Signal Recovery  Problems} \label{sec:motivating-examples}

Semidefinite programs of form \eqref{convex-problem} are useful for signal recovery problems.  Two examples are phase retrieval and the recovery of the sparsest element in a subspace.  
\begin{alignat}{2}
&\text{Phase retrieval:} \qquad & &\text{Given: } |\< \x_0, \zi \> |^2 \text{ for known } \zi \in \R^n \label{phase-retrieval-problem}\\
&&&\text{Find: } \x_0. \notag \\
&&&\ \notag \\
&\text{Sparsest element in a subspace:} \qquad & &\text{Given: } \text{arbitrary basis of } W \subset \R^n  \label{subspace-problem}\\
&&&\text{Find: } \x_0\text{, the sparsest nonzero element in } W. \notag
\end{alignat}
Phase retrieval is the task of recovering a vector from only the amplitudes of linear measurements of it. It has applications in X-ray crystallography \cite{CESV2011}.  
The problem of finding the sparsest element in a subspace has applications to dictionary learning \cite{SWW2012} and blind source separation \cite{ZP2001}.  

Both of these problems are nonconvex,  and they can be relaxed to a semidefinite program with a procedure known as lifting.  Instead of directly seeking a vector $\x$, one seeks a lifted matrix $\X$ that is a proxy for $\x \x^t$.  Quadratic measurements on $\x$ become linear on $\X$, and the desired $\X$ is positive-semidefinite.  Lifting takes a nonconvex vector recovery problem and replaces it with an $n\times n$ semidefinite program.  The vector $\x$ can be estimated by computing the leading eigenvector of the optimal $\X$.  

The semidefinite relaxation of \eqref{phase-retrieval-problem} can be derived as follows.  The quadratic measurements can written as $b_i = | \< \x, \zi \> |^2 = \zi^t \x \x^t \zi = \zi^t \X \zi = \langle \X, \zi \zi^t \rangle,$ which are linear in $\X$.  We enforce the positive-semidefinite constraint because the desired $\X$ is $\x\x^t \matrixgeq 0$.  Further, as we seek a rank-one $\X$, we attempt to minimize the rank among all positive-semidefinite matrices consistent with the data.  Doing so is NP-hard \cite{CP1986}, so we instead minimize the sum of the singular values (nuclear norm) of $\X$, which is a convex relaxation of rank.  For positive-semidefinite matrices, the nuclear norm equals the trace of the matrix.  Hence, \eqref{phase-retrieval-problem} can be relaxed to the semidefinite program known as PhaseLift \cite{CESV2011, CSV2011}:
\begin{align}
\min \trace(\X) \text{ s.t. } \X \matrixgeq 0, \langle \X, \zi \zi^t \rangle = b_i. \label{phaselift}
\end{align}
Here and henceforth, a constraint with a  subscript is to be interpreted as imposing the corresponding constraint for all values of the subscript.  If we additionally seek a sparse $\x$ (and hence sparse $\X$), we could consider the compressive variant of PhaseLift  \cite{LV2012, OYDS2012}:
\begin{align}
\min \trace(\X) + \mu \|\X\|_1 \text{ s.t. } \X \matrixgeq 0, \langle \X, \zi \zi^t \rangle = b_i. \label{compressive-phaselift}
\end{align}
where $\mu$ is a scalar parameter and $\|\X\|_1$ is the $\ell_1$ norm of the vectorization of $\X$.

Our second motivating signal reconstruction problem is the task of finding the sparsest element in a subspace $W$.  Because zero is trivially the sparsest element, we need to introduce a normalization.  A natural choice is to search for the sparsest vector with unit length in $\ell_2$.  Using the $\ell_1$ norm as a proxy for sparsity, we consider the nonconvex problem
\begin{align}
\min \|\x\|_1 \text{ s.t. } \x \in W, \|\x\|_2 = 1. \label{min-l1-l2}
\end{align}
A  semidefinite relaxation of \eqref{min-l1-l2} is as follows.  Letting $\X$ again be a proxy for $\x\x^t$, the $\ell_2$ normalization becomes $\trace(\X)=1$.  Let $\|\X\|_1$ be the $\ell_1$ norm of the vectorization of $\X$.  Let $\{\z_i\}$ be a basis for $W^\perp$.  Then, the subspace constraint can be written as $\z_i^t \x = 0$.  After lifting, this constraint could be written as $\z_i^t \X \z_i  = \langle \X, \zi \zi^t \rangle = 0$.  Hence, the subspace problem \eqref{subspace-problem} can be relaxed into the semidefinite program  
\begin{align}
\min \|\X\|_1 \text{ s.t. } \X \matrixgeq 0,\langle \X, \zi \zi^t \rangle= 0, \trace(\X) = 1. \label{subspace-lifted}
\end{align}

The programs \eqref{phaselift}, \eqref{compressive-phaselift}, and \eqref{subspace-lifted} are all examples of rank-one semidefinite programs from signal recovery.  As we will discuss, problems \eqref{compressive-phaselift} and \eqref{subspace-lifted} exhibit markedly different behavior concerning dual certificates.  Under appropriate random models for $\zi$, dual certificates will exist at minimizers of \eqref{compressive-phaselift} but not of \eqref{subspace-lifted}.



\subsection{Dual Certificates and Strong Duality}

A common method for proving that a particular $\X_0$ is a minimizer of a matrix recovery problem is certification  \cite{G2011, CSV2011, CL2012, DH2012, LV2012, ARM2012}.  A certificate of optimality at $\X_0$ defines a hyperplane that separates the feasible set of parameters from the set of points with smaller objectives.  Under mild assumptions on $f$, a certificate necessarily exists at a minimizer.  Finding a certificate guarantees that $\X_0$ is a minimizer.  For the claims in this section, $\X_0$ can be of arbitrary rank.  

In many cases, the certificate can be represented in terms of  the problem's Lagrangian dual variables.  Such certificates are known as dual certificates.  
The Lagrangian of \eqref{convex-problem} is given by
\begin{align}
\mathcal{L}(\X,  \blambda, \Q) = f(\X)  + \langle \blambda, \A(\X) - \bfb \rangle + \langle \Q, \X\rangle. \nonumber
\end{align}
where $\Q \in \Rsymn$ and $\blambda \in \R^m$. The Lagrangian dual problem for  \eqref{convex-problem} is
\begin{align}
\sup_{\Q\matrixleq 0, \blambda} \inf_\X \mathcal{L}(\X, \blambda, \Q) \label{dual-problem}
\end{align}
 The dual variables $(\blambda, \Q)$ are  dual-feasible when $\Q \matrixleq 0$.
Let $p^*$ and $d^*$ be the optimal values of \eqref{convex-problem} and \eqref{dual-problem}, respectively.  The duality gap is $p^* - d^*$, which is always nonnegative.  Problem \eqref{convex-problem} is said to satisfy strong duality when the duality gap is zero and the dual optimum is attained.

We call $(\blambda, \Q)$ a dual certificate at $\X_0$ if
\begin{align}
\A^* \blambda + \Q &\in - \partial f(\X_0), \label{optimality-condition} \\
\Q &\matrixleq 0, \Q \perp \X_0. \label{slackness-condition}
\end{align}
That is, a dual certificate $(\blambda, \Q)$ solves the KKT optimality conditions \cite{BV2004}.  Observe that the optimality condition \eqref{optimality-condition} ensures that $0$ is in the subdifferential of $\mathcal{L}$ with respect to $\X$ at $\X_0$.  Conditions \eqref{slackness-condition} enforce dual-feasibility and complementary slackness.
We will sometimes refer to $\Y = \A^* \blambda + \Q$ as a dual certificate. 

By elementary arguments from convex optimization,  a dual certificate at $\X_0$ certifies that $\X_0$ is a minimizer.  Further, existence of a dual certificate is equivalent to strong duality.  

\begin{theorem} \label{theorem:dual-certificate-sufficiency}
If $(\blambda, \Q)$ is a dual certificate at $\X_0$, then $\X_0$ is a minimizer to \eqref{convex-problem}.
\end{theorem}

\begin{theorem} \label{theorem:dual-certificate-equivalence}
Let $\X_0$ be a minimizer of \eqref{convex-problem}.  The following are equivalent:
\begin{enumerate}
\item[(a)] $(\blambda, \Q)$ is a dual certificate at $\X_0$,
\item[(b)] $(\blambda, \Q)$ is dual optimal and strong duality holds.
\end{enumerate}

\end{theorem}

\subsection{Dual Certificates May Not Exist}

For some semidefinite problems of form \eqref{convex-problem}, a dual certificate fails to exist at a minimizer.  
We provide two examples.  Let $\e_i$ be the $i$th standard basis vector, let $\q \otimes \y = \q \y^t + \y \q^t$ be the symmetric tensor product, let $\|\X\|_\text{F}$ be the Frobenius matrix norm.  We note the important technical fact that 
\begin{align}\X \matrixgeq 0 \text{ and }  \langle \X, \q\q^t \rangle &= 0 \text{ for } \q \in \R^n \Longrightarrow  \X\q = 0 
 \Longrightarrow  \langle \X, \q \otimes \e_j \rangle = 0 \ \forall \ j. \label{psd-property} 
\end{align}
That is, if a positive-semidefinite matrix is orthogonal to another positive-semidefinite matrix, there are additional linearly independent orthogonal measurements that hold automatically.  As we will see, these additional measurements play an important role in dual certificate existence.

\begin{example} \label{example:strong} The minimizer and only feasible point for
\begin{align}
\min \frac{1}{2} \|\X\|_\text{F}^2 \text{ subject to } \X \matrixgeq 0, \left \langle \X, \begin{pmatrix} 0 & 0 \\ 0 & 1 \end{pmatrix} \right \rangle &= 0, \left \langle \X, \begin{pmatrix} 1 & 1 \\ 1 & 1 \end{pmatrix} \right \rangle = 1 \label{counterexample} 
\end{align}
is $\X_0= \e_1 \e_1^t$.  Further, there is no dual certificate at $\X_0$ for this problem.
\end{example}

By \eqref{psd-property}, any feasible $\X$ for Example \ref{example:strong} satisfies $\< \X, \e_1 \otimes \e_2 \> = 0$. Hence, the minimizer and only feasible point of \eqref{counterexample} is $\X_0$.
The subdifferential of $f(\X) = \frac{1}{2} \| \X \|_\text{F}^2$ is $\partial f( \X_0) = \{ \X_0 \}.$
We note that there is no dual certificate because there is no $(\blambda, \Q)$ satisfying  $ \Q \matrixleq0, \Q \perp \X_0$,  and
$$
- \begin{pmatrix} 1 & 0 \\ 0 & 0 \end{pmatrix} =  \lambda_1 \begin{pmatrix}0 & 0 \\ 0 & 1 \end{pmatrix} + \lambda_2 \begin{pmatrix} 1 & 1 \\ 1 & 1 \end{pmatrix} + \Q.
$$
In this problem, one can show that there is no duality gap.  Hence, the dual optimum is not attained.  

Note that \eqref{psd-property} provides the additional measurement $\< \X, \e_1 \otimes \e_2 \> = 0$.  If this constraint were included in \eqref{counterexample}, then the program would be equivalent (in the sense that the feasible set is unchanged), yet a dual certificate would exist at $\X_0$.  This is a simple case of the completion process that will be described in Section \ref{section:completion-process}.

We now provide a more involved example from the problem of finding the sparsest element in a subspace.
\begin{example} \label{example:sparsest}
Let $\x_0 \in \R^n$ be such that it does not have constant magnitude on its support.  Let $\z_i \perp \x_0$, for $i = 1 \ldots m$ and $m\geq 1$.  There is no dual certificate at $\x_0 \x_0^t$ for
\begin{align}
\min \|\X\|_1 \text{ s.t. } \X \matrixgeq 0,\langle \X, \zi \zi^t \rangle= 0, \trace(\X) = 1. \label{example-sparsest-elt}
\end{align}
\end{example}
This example is a semidefinite relaxation for the problem of finding the sparsest element in the subspace orthogonal to $\Span\{\zi\}$.  By construction, $\x_0$ is in this space.  In some cases, if $\x_0$ is sparse enough, then $\x_0 \x_0^t$ will minimize \eqref{example-sparsest-elt}.  Example \ref{example:sparsest} establishes that even if $\x_0 \x_0^t$ is a minimizer, a dual certificate will generally fail to exist.  Hence, it it is futile to attempt a recovery proof based on constructing a dual certificate for the problem \textit{as written}.  In Section \ref{section:discussion}, we will provide an equivalent formulation that guarantees the existence of dual certificates at minimizers.

The proof of the claim in Example \ref{example:sparsest} is as follows.  By \eqref{optimality-condition}--\eqref{slackness-condition}, a dual certificate is such that
\begin{align}
\lambda_0 \I + \sum_{i = 1}^m \lambda_i \zi \zi^t + \Q &\in -\partial \|\cdot\|_1(\X_0) \label{dual-cert-cond-l1} \\
\Q &\matrixleq0, \Q \perp \x_0 \x_0^t
\end{align}
Letting $\Omega$ be the support of $\x_0$, note that the support of $\x_0 \x_0^t$ is $\Omega \times \Omega$.  Let $\xoo$ be the restriction of $\x_0$ to its support.  Note that the restriction to $\Omega \times \Omega$ of any subgradient of $\|\cdot \|_1$ at $\x_0 \x_0^t$ is $\sgn (\x_0 \x_0^t)$.  By computing the restriction onto $\Omega$ of the right multiplication of \eqref{dual-cert-cond-l1} by $\x_0$, we get 
\begin{align}
\lambda_0 \xoo = - \|\x_0\|_1 \sgn \xoo
\end{align}
This equality is impossible if $\xoo$ is not of constant magnitude.   Hence a dual certificate does not exist at $\X_0$.  


\subsection{Constraint Qualifications}

The two examples of the prior section illustrate the well known fact  that semidefinite programs may not satisfy strong duality \cite{RTW1997, LW2012, F1994}.  A constraint qualification (CQ) is a condition under which strong duality is ensured.  For example, the presence of a strictly feasible $\X \matrixgt 0$ such that $\A(\X) = b$ is a constraint qualification and is known as Slater's condition \cite{BV2004}.  
Slater's condition can fail for low-rank matrix recovery problems.    For instance, the orthogonality constraints in Examples \ref{example:strong} and \ref{example:sparsest} ensure that there are no strictly feasibility points.  

The work in this paper will be based off the following constraint qualification.  The Rockafellar-Pshenichnyi condition \cite{LW2012, G69, W80, Psh71}  in the present context is that $\X_0$ minimizes \eqref{convex-problem} if and only if there exists a $Y \in (- \partial f(\X_0)) \cap \partial \indposaxb(\X_0)$, where $\indposaxb$ is the indicator function of the feasible set.  That is, a certificate (which may or may not be a dual certificate) exists at any minimizer and belongs to the subdifferential of the indicator function of the feasible set.  
The set of candidate dual certificates given by \eqref{optimality-condition}--\eqref{slackness-condition} is a cone, which we will denote by $S$.  Observe that
\begin{align}
S &:= \left \{ \sum_i \lambda_i \bfA_i + \Q \mid \Q \matrixleq 0, \Q \perp \X_0 \right \} 
= \partial \indpos(\X_0) + \partial \indaxb(\X_0),  \label{s:definition}
\end{align}
where the second equality follows directly from the definition of the sub\-differential of indicator functions.  Thus, if 
\begin{align}
\partial \indpos(\X_0) + \partial \indaxb(\X_0) = \partial \indposaxb(\X_0), \label{CQ-subgradient-additivity}
\end{align} 
then the certificate guaranteed by Rockafellar-Pshenichnyi is necessarily a dual certificate.  Put differently \eqref{CQ-subgradient-additivity} is a constraint qualification.  If it holds, then a dual certificate exists at any minimizer.  It is known as a weakest constraint qualification \cite{LW2012} because it is independent of the objective $f$.

\subsection{Main Results}

In this paper, we interpret the Rockafellar-Pshenichnyi constraint qualification in the context of rank-one matrix recovery problems.  We present a completeness condition on the measurement matrices $\bfA_i$ such that strong duality holds.  For the case when the condition fails, we will also present a completion process that adds measurements to \eqref{convex-problem} such that the feasible set is unchanged and the completeness condition holds.  In this equivalent problem, a dual certificate necessarily exists.

\subsubsection{Completeness conditions for dual certificate existence}

The lack of a dual certificate in Examples \ref{example:strong} and \ref{example:sparsest} happens because there are measurement matrices $\bfA_i$ that are positive-semidefinite and orthogonal to $\X_0$.  If this case is excluded, a dual certificate necessarily exists at the rank-one solution $\X_0$.

\begin{theorem} \label{theorem:specific}
If $\X_0=\x_0 \x_0^t$ minimizes \eqref{convex-problem} and there does not exist a nonzero $\bfA \in \Span\{\bfA_i\}_{i=1}^m $ such that $\bfA\matrixgeq 0$ and $\bfA \perp \X_0$, then there exists a dual certificate at $\X_0$. 
\end{theorem}

\noindent This theorem will be proved as a special case of Theorem \ref{theorem:general}.  It can also be proved by noting that the Theorem of the Alternative in Section 5.9.4 of  \cite{BV2004} shows that Slater's condition holds.

For some measurement matrices $\bfA_i$ of practical interest, Theorem \ref{theorem:specific} is applicable.  Consider the compressive phase retrieval problem \eqref{compressive-phaselift}.  We adopt the model of  \cite{LV2012} and consider the case where $\{ \z_i\}_{i = 1 \ldots m}$ are i.i.d. Gaussian with $m \leq n$.  In this case, $\bfA_i = \zi \zi^t$.  The following corollary implies that a dual certificate exists at $\X_0 = \x_0 \x_0^t$ with probability 1.
\begin{corollary} \label{corollary:independence}
If $\X_0 = \x_0 \x_0^t$ minimizes \eqref{convex-problem} and $\{\bfA_i \x_0\}_{i=1}^m$ are independent, then there exists a dual certificate at $\X_0$.
\end{corollary}

We now present a more general sufficient condition for dual certificate existence.  
If there is a positive-semidefinite measurement matrix $\bfA$ that is orthogonal to $\X_0$, then \eqref{psd-property} provides additional constraints on $\X$ that may or may not be implied by the linear constraints $\A(\X)=\bfb$ alone.  For any $\q\in \range(\bfA)$  and for any $\y$, all feasible $\X$ satisfy  $\langle \X, \y \otimes \q \rangle = 0$.  If these measurement matrices are included in $\A$, then we consider the set of measurements to be complete.  Recalling the definition that $S := \left \{ \sum_i \lambda_i \bfA_i + \Q \mid \Q \matrixleq 0, \Q \perp \X_0 \right \}$, we say that $\A$ is complete at $\X_0 = \x_0 \x_0^t$ if the following condition holds.
\begin{align}
&\text{Completeness condition:} \notag\\
&\text{If } \bfA = \A^* \blambda \matrixgeq 0, \bfA \perp \x_0 \x_0^t \text{, then } \y \otimes \q \in S \text{ $\forall$ } \y \text{ and $\forall$ $\q \in  \range(\bfA)$.} \label{complete-condition}
\end{align}
Note that for this condition, it suffices, but is not necessary, that $\y \otimes \q$ belongs to $\range(\A^*)= \Span\{\bfA_i\}$.  Technically, the condition only requires that $\y \otimes \q$ differs from this range by something negative-semidefinite.  For ease of exposition, we will sometimes say that $S$ is complete at $\X_0$ if \eqref{complete-condition} holds.   The completeness condition can equivalently be written as
\begin{align}
\text{If }\q \q^t \in S \text{ and } \q \perp \x_0, \text{ then }  \y \otimes \q \in S \text{ for all } \y.
\end{align}

The main theorem is as follows.  

\begin{theorem}\label{theorem:general}
Let $\X_0 = \x_0 \x_0^t$ minimize \eqref{convex-problem}.  If  $\A$ satisfies the completeness condition \eqref{complete-condition}
then strong duality holds and a dual certificate exists at $\X_0$.
\end{theorem}

Roughly, the theorem states that if all of the linear measurements implied by $\X \matrixgeq 0$ and $\mathcal{A}(\X) = \bfb$ are included, then strong duality holds and a dual certificate exists.  

\subsubsection{Completion process} \label{section:completion-process}

Some programs of form \eqref{convex-problem} do not satisfy the completeness condition \eqref{complete-condition}.  Hence, a dual certificate may fail to exist at minimizers.  For such problems, there exists an equivalent program which satisfies \eqref{complete-condition}.  It can be constructed by a process that completes the measurement matrices in $\A$.  That is, an optimality certificate for any program \eqref{convex-problem} can be expressed as a dual certificate for the problem augmented with linear constraints implied by $\X\matrixgeq0$ and $\A(\X)=\bfb$. 

\begin{corollary} \label{corollary:completion}
If $\X_0 = \x_0 \x_0^t$ minimizes \eqref{convex-problem}, then there exists an equivalent equivalent problem 
\begin{align}
\min f(\X) \text{ such that } \X \matrixgeq 0, \tilde{\A}(\X) = \tilde{\bfb}
\end{align}
such that there exists a dual certificate at $\X_0$.  This problem is equivalent to \eqref{convex-problem} in the sense that the sets $\{ \X \matrixgeq 0, \A(\X) = \bfb \}  \ and\  \{ \X \matrixgeq 0, \tilde{\A}(\X) = \tilde{\bfb} \}$ are equal.
\end{corollary}

The following procedure outlines a theoretical process to complete the set of measurement matrices $\{\bfA_i\}$ in order to  satisfy the completeness condition \eqref{complete-condition}.  Given $\A$ consisting of measurement matrices $\{\bfA_i\}$, build $\tilde{\A}$ as follows:

\begin{enumerate}
\item Consider all $ \bfA\matrixgeq 0, \bfA \in \Span \{\bfA_i\}, \langle \bfA, \X_0 \rangle = 0$.
\item Write each $\bfA = \sum_k c_k \q_k \q_k^t$ with $c_k>0$.
\item For every $j$, if $\q_k \otimes \e_j \notin \Span \{\bfA_i\}$, append $\q_k \otimes \e_j$ to $\{\bfA_i\}$.  
\item Repeat until $\{\bfA_i\}$ remains unchanged.
\end{enumerate}
Let $\tilde{\A}$ be the operator corresponding to $\{\bfA_i\}$ upon termination of this process.  Note that step 3 corresponds to appending $\langle \X, \q_k \otimes \e_j \rangle = 0$ to $\A(\X) = \bfb$.  
This process will produce an $\A$ satisfying \eqref{complete-condition}, and it will terminate after finitely many repetitions because  $\text{dim}(\Span\{\bfA_i\} )$ increases at each repetition.  Because the resulting $\tilde{\A}$ will satisfy \eqref{complete-condition}, we apply Theorem \ref{theorem:general} and have thus proven Corollary \ref{corollary:completion}.  The semidefinite feasibility problem implicit in the first step is of unknown computational complexity \cite{Ramana1995}.  Hence, this procedure is of limited computational use.  See \cite{CSW2013} for computational preprocessing and regularization of   semidefinite programs that fail Slater's condition.

\subsubsection{Weak necessity of the completeness condition}

If the measurement matrices fail to satisfy the completeness condition \eqref{complete-condition}, a particular problem of form \eqref{convex-problem} may or may not have a dual certificate at a minimizer $\X_0$.  Nonetheless, there is a problem for which a dual certificate does not exist.  

\begin{theorem} \label{theorem:necessity}
Fix $\X_0 = \x_0 \x_0^t$ and the matrices $\{\bfA_i\}_{i=1}^m$.  If $\A$ does not satisfy the completeness condition \eqref{complete-condition} at $\X_0$, there exists a convex problem \eqref{convex-problem} such that $\X_0$ is a minimizer and a dual certificate does not exist at $\X_0$.
\end{theorem}

This weak form of necessity of the completeness condition arises  because of an equivalence between completeness and the additivity of subgradients of indicator functions over the constraints, as proven in Lemma \ref{lemma:subgradient-additivity}.

\subsection{Discussion} \label{section:discussion}

The important message of this paper is that orthogonal measurements with respect to positive-semidefinite matrices can give rise to situations where dual certificates do not exist (and strong duality does not hold) for semidefinite programs.    If a semidefinite program involves such measurements, there are additional measurements that should be included when building a dual certificate.  For example, if $\A(\X) = \bfb$ includes the measurement $\langle \X, \q \q^t \rangle = 0$, then $\langle \X, \q \e_j^t + \e_j \q^t \rangle = 0$ should be appended for all $j$, unless they are already implied by $\A(\X) = \bfb$.  

The sparsest element problem provides an important example of when care must be taken in posing a semidefinite relaxation.  As discussed in Section \ref{sec:motivating-examples}, one semidefinite relaxation of the task of finding the sparsest unit vector in a subspace $W$ is
\begin{align}
\min \|\X\|_1 \text{ s.t. } \X \matrixgeq 0, \langle \X, \zi \zi^t \rangle = 0, \trace(\X) = 1, \label{subspace-lifted-vXv}
\end{align}
where $\{ \z_j \}$ forms a basis of $W^\perp$.  As proven in Example \ref{example:sparsest}, a dual certificate can not exist at any $\x_0 \x_0^t$ such that $\x_0$ has nonconstant magnitudes on its support.  Hence, it is futile to attempt to prove that $\x_0 \x_0^t$ minimizes \eqref{subspace-lifted-vXv} for general sparse $\x_0$ by a dual certificate argument for \eqref{subspace-lifted-vXv} \textit{as written}.   We observe that the completeness condition \eqref{complete-condition} is violated for \eqref{subspace-lifted-vXv}.  If we follow the regularization procedure from Section \ref{section:completion-process}, we arrive at the equivalent program
\begin{align}
\min \|\X\|_1 \text{ s.t. } \X \matrixgeq 0, \langle \X, \zi \e_j^t + \e_j \zi^t \rangle = 0, \trace(\X) = 1. \label{subspace-lifted-Xve}
\end{align}
The program \eqref{subspace-lifted-Xve} satisfies the completeness condition \eqref{complete-condition}, and hence a dual certificate exists at a minimizer $\x_0 \x_0^t$.  For emphasis, we remark the the semidefinite programs \eqref{subspace-lifted-vXv} and \eqref{subspace-lifted-Xve} are equivalent, yet dual certificates exist at minimizers of the latter but not the former.

We remark that the completeness condition \eqref{complete-condition} is only a sufficient condition for existence of dual certificates.  For any particular problem, it may not be necessary.  It is, however, necessary for some particular problem, as per Theorem \ref{theorem:necessity}. 

We caution the reader of a subtlety of semidefinite programs of form \eqref{convex-problem}.  Consider such a program that satisfies strong duality.  Appending generic measurements to $\A(\X) = \bfb$ results in a program that may or may not satisfy strong duality.  This is because these additional measurements may cause the completeness condition \eqref{complete-condition} to be unsatisfied.  In this case, Theorem \ref{theorem:general} is not applicable, and the program may or may not satisfy strong duality as written.  To guarantee strong duality, a completion process like above is needed to ensure \eqref{complete-condition} holds.

We now place the completion procedure from Section \ref{section:completion-process} in context. The completion process could be viewed as regularizing the semidefinite program \eqref{convex-problem}.  Regularization is the modification of a semidefinite program or its dual in order to ensure strong duality.  One approach for this is a minimal cone regularization  \cite{BW1981, RTW1997}, where the conic constraint is modified.  Another approach is the extended Lagrange-Slater Dual (ELSD), which is an alternative to the Lagrangian dual \cite{Ramana1995, RTW1997}.  It can be constructed with polynomially many additional variables.  The regularization procedure in the present paper is different from these alternatives because it attains strong duality without changing the structure of the program.  The conic constraint and overall form remain the same, as only additional measurements are added.  The procedure can not be written down mechanically; hence, it is not suitable for numerical computations.  Instead, its simplicity in form makes it more useful for analytical constructions of dual certificates.

As stated, Theorem \ref{theorem:general} is proven when the minimizer $\X_0$ has rank one.  It is an interesting problem to see if a corresponding result holds in the case of low rank $\X_0$.  Because this paper is motivated by vector recovery problems that are lifted to rank-one matrix recovery problems, this extension is left for future work.

\subsection{Organization of this paper}
In Section \ref{section:notation}, we present the notation used throughout the paper.
In Section \ref{section:background-theorems}, we prove Theorems \ref{theorem:dual-certificate-sufficiency} and \ref{theorem:dual-certificate-equivalence} which are elementary results from convex optimization.
In Section \ref{section:proof-of-main-theorems}, we prove Theorems \ref{theorem:specific} and \ref{theorem:necessity} and Corollary \ref{corollary:independence}.  Corollary \ref{corollary:completion} was proven in Section \ref{section:completion-process}.  
The proofs of Theorems \ref{theorem:specific} and \ref{theorem:general} rely on technical lemmas concerning the additivity of subdifferentials of indicator functions, and on the closedness of $S$.  These lemmas are proven in Section \ref{section:lemmas}.

\subsection{Notation} \label{section:notation}
Let $\Rsymn$ be the space of symmetric, real-valued $n \times n$ matrices.  Matrices will be denoted with boldface capital letters, and vectors will be denoted with boldface lowercase letters.  Let $\X \matrixgeq 0$ denote that $\X$ is positive-semidefinite.   Let $\<\cdot, \cdot\>$ be the usual inner product for vectors and the Hilbert-Schmidt inner product for matrices.  Let $\x \otimes \y = \x\y^t + \y\x^t$ be the symmetric tensor product. For a subspace $V \subset \R^n$, let $V^\perp$ be the orthogonal complement with respect to the ordinary inner product.  Let $\ivp$ be the matrix corresponding to orthogonal projection of vectors onto $V^\perp$.  Let $\projvp \X =\ivp \X \ivp$ be the projection of symmetric matrices onto symmetric matrices with row and column spans in $V^\perp$. Let $\projxop$ be the special case in the instance where $V = \Span \{\x_0\}$.  In the special case where $\x_0$ is the coordinate basis element $\e_1$,  $\projxop \X$ is the projection of $\X$ to the lower-right $n-1 \times n-1$ block.

Let the indicator function for the set $\Omega$ be $\ind_\Omega(\X) = \begin{cases}0 & \text{ if } \X \in \Omega, \\ + \infty & \text{ if } \X \notin \Omega. \end{cases}
$

\section{Proofs of Background Theorems \ref{theorem:dual-certificate-sufficiency} and \ref{theorem:dual-certificate-equivalence}} \label{section:background-theorems}

Theorems \ref{theorem:dual-certificate-sufficiency} and \ref{theorem:dual-certificate-equivalence} follow from classical arguments in convex optimization.

\begin{proof}[Proof of Theorem \ref{theorem:dual-certificate-sufficiency}]
Consider a feasible $\X$.  Because $\A^* \blambda + \Q \in - \partial f(\X_0)$, 
\begin{align*}
f(\X) - f(\X_0) &\geq - \< \A^* \blambda + \Q, \X - \X_0 \> = -\<\Q, \X \> \geq 0,
\end{align*}
where the equality uses $\A(\X)=\A(\X_0)$ and $\Q \perp \X_0$, and the second inequality uses $\Q \matrixleq 0$ and $\X \matrixgeq 0$. 
\end{proof}

\begin{proof}[Proof of Theorem \ref{theorem:dual-certificate-equivalence}]
To prove (a) $\Rightarrow$ (b), we observe that \eqref{optimality-condition} implies $ 0 \in \partial_\X \mathcal{L}(\X_0, \blambda, \Q)$.  Hence, $\X_0$ minimizes $\mathcal{L}(\X, \blambda, \Q)$ over $\X$.  Hence $g(\blambda, \Q) = \mathcal{L}(\X_0, \blambda, \Q)$.   By slackness and feasibility of $\X_0$, $\mathcal{L}(\X_0, \blambda, \Q) = f(\X_0) = p^*$.  Hence $(\blambda, \Q)$ is dual optimal and strong duality holds.

To prove (b) $\Rightarrow$ (a), we observe that strong duality and dual optimality of $(\blambda, \Q)$ imply
\begin{align}
f(\X_0) = \inf_\X f(\X)  + \<\blambda, \A (\X) - b \> + \< \Q, \X \> \label{strong-duality-lagrangian}
\end{align}
In particular,$f(\X_0) \leq f(\X_0)  + \<\blambda, \A (\X_0) - b \> + \< \Q, \X_0 \>$, which implies $\< \Q, \X_0\> \geq 0$ by feasibility of $\X_0$.  By dual feasibility, $\Q \matrixleq 0$ and hence $\< \Q, \X_0 \> \leq 0$. We thus have $\Q \perp \X_0$.  The infimum in \eqref{strong-duality-lagrangian} is achieved by $\X_0$.  Hence, $0 \in \partial_\X \mathcal{L}(\X_0, \blambda, \Q)$, and we conclude $\A^* \blambda + \Q \in -\partial f(\X_0)$.
\end{proof}

\section{Proofs of Main Results} \label{section:proof-of-main-theorems}

In this section, we present the proofs of the main theorems and Corollary \ref{corollary:independence}.    

\subsection{Proof of Theorems \ref{theorem:specific} and \ref{theorem:general} and Corollary \ref{corollary:independence}}
Under the assumptions of Theorem \ref{theorem:specific}, the set $S$ trivially satisfies the completeness condition \eqref{complete-condition}.  The theorem is thus a special case of Theorem \ref{theorem:general}, and we will prove them together.

The strategy of proof involves rewriting \eqref{convex-problem} in an unconstrained form.  Existence
of a dual certificate is guaranteed when the subdifferentials of the sum of two indicator functions is the sum of their respective subdifferentials.  In Lemma \ref{lemma:subgradient-additivity}, we use a separating hyperplane argument to prove additivity of these subdifferentials under the condition \eqref{complete-condition}.

\begin{proof}[Proof of Theorems \ref{theorem:specific} and \ref{theorem:general}] We first rewrite the problem \eqref{convex-problem} without constraints.  $\X_0$ minimizes \eqref{convex-problem} if and only if $\X_0$ minimizes the problem 
\begin{align}
\min f(\X) + \indposaxb(\X), \label{convex-problem-no-constraints}
\end{align}
which, by convexity, happens if and only if
\begin{align}
0 \in \partial( f + \indposaxb ) (\X_0).  \label{zero-in-subgradient-sum}
\end{align}
By assumption, $f$ is continuous everywhere.  Hence, the Moreau-Rockafellar Theorem \cite{Rock1970} guarantees that 
\begin{align}
\partial(f + \indposaxb) (\X_0) = \partial f(\X_0) + \partial \indposaxb(\X_0). \label{moreau-rockefellar-consequence}
\end{align}
By  Lemma \ref{lemma:subgradient-additivity}, the completeness condition \eqref{complete-condition} is equivalent  to
\begin{align}
\partial \indposaxb (\X_0)= \partial \indpos (\X_0) + \partial \indaxb(\X_0). \label{additivity-of-subgradients}
\end{align}
We note that
\begin{align}
\partial \indpos(\X_0) &= \{\Q \mid \Q \matrixleq 0, \Q \perp \X_0 \},\\
\partial \indaxb(\X_0) &= \left \{ \A^* \blambda \right \},\\
\partial \indpos(\X_0) + \partial \indaxb(\X_0) &=  \left \{\A^* \blambda + \Q \mid \Q \matrixleq 0, \Q \perp \X_0 \right\} = S. \label{form-of-sum-of-subgradients}
\end{align}
We conclude there exists a dual certificate $(\blambda, \Q)$ by combining \eqref{zero-in-subgradient-sum}, \eqref{moreau-rockefellar-consequence}, \eqref{additivity-of-subgradients}, and \eqref{form-of-sum-of-subgradients}.
\end{proof}

The corollary follows from Theorem \ref{theorem:specific} because the independence assumption implies that there are no nontrivial linear combinations of measurement matrices that are positive-semidefinite and orthogonal to $\X_0$.
\begin{proof}[Proof of Corollary \ref{corollary:independence}]
Consider $\bfA \matrixgeq 0$, $\bfA = \sum_i \lambda_i \bfA_i$, $\bfA \perp \x_0 \x_0^t$.  By \eqref{psd-property},  $\bfA \x_0 = 0$.  Hence, $\sum_i \lambda_i \bfA_i \x_0 = 0$.  By independence of $\{\bfA_i \x_0\}$, $\lambda_i = 0$ for all $i$.  Hence the conditions of Theorem \ref{theorem:specific} are met and there exists a dual certificate at $\X_0$.  

\end{proof}

\subsection{Proof of Theorem \ref{theorem:necessity}}

Theorem \ref{theorem:necessity} provides a weak form of necessity for the completeness condition \eqref{complete-condition}.  If $-\partial f(\X_0)$ only contains matrices that are not of form $S$,  there will be no dual certificate.  When $S$ does not satisfy \eqref{complete-condition}, there is a matrix orthogonal to all feasible points, and we choose $f$ to have a gradient in the opposite direction.  This argument also plays an important role in the primary technical lemma establishing equivalence between the completeness condition and additivity of subgradients. 

\begin{proof}[Proof of Theorem \ref{theorem:necessity}]
If $S$ does not satisfy the completeness condition \eqref{complete-condition}, then there is a $\q \perp \x_0$ and a $\y$ such that $\q \q^t \in S$ and $\y \otimes \q \notin S$.  Consider the problem 
\begin{align}
\min \< -\y \otimes \q, \X \> \text{ subject to } \X \matrixgeq 0, \A(\X) = \A(\X_0). \label{example-necessity}
\end{align}
First, we show that $\X_0$ is a minimizer. Because $\q\q^t \in S$, $\q\q^t = \A^* \blambda + \Q$ for some $\Q \matrixleq 0$, $\Q \perp \X_0$.  Hence, 
$$
\<\X, \q\q^t\> = \<\X-\X_0, \q\q^t \> = \< \X-\X_0, \A^* \blambda + \Q \> = \<\X - \X_0, \Q\>  = \<\X, \Q\> \leq 0,
$$
where the third equality uses $\A(\X-\X_0) = 0$. Because all feasible $\X$ are positive-semidefinite, we observe that $\< \X, \q\q^t \> \geq 0$ and conclude that  $\X \perp \q \q^t$.  Hence, $ {\< -\y \otimes \q, \X \>  = 0}$ for any feasible $\X$.  Hence $\X_0$ is a minimizer.  

There is no dual certificate at $\X_0$ because $-\partial f(\X_0)$ contains the single element $\y \otimes \q \notin S.$
\end{proof}

\section{Technical Lemmas} \label{section:lemmas}

The main technical lemma establishes that the completeness condition is equivalent to the additivity of subgradients of a class of indicator functions.

\begin{lemma} \label{lemma:subgradient-additivity}
Let $\X_0 = \x_0 \x_0^t$ and $\A(\X_0)=b$.  The cone $S = \{ \A^*\blambda + \Q \mid \Q \matrixleq 0, \Q \perp \X_0\}$ satisfies the completeness condition \eqref{complete-condition} if and only if
\begin{align}
\partial \indposaxb (\X_0) = \partial \indpos (\X_0) + \partial \indaxb(\X_0). \label{subgradient-additivity}
\end{align}
\end{lemma}
We recall that $S := \partial \indpos(\X_0) + \partial \indaxb(\X_0)$.  
One direction of the proof follows from the same argument as the proof of Theorem \ref{theorem:necessity}.  The other direction follows by showing  $\Y \not\in S \Rightarrow \Y \notin \partial \indposaxb(\X_0)$.  To show $\Y$ is not such a subgradient, we use a separating hyperplane argument.  That argument requires that $S$ is closed, as proven in Lemma \ref{lemma:closed}.  It also  hinges on Lemma \ref{lemma:positive-perturbation} which classifies when perturbations from a rank-one $\X_0$ remain positive-semidefinite.  The completeness condition is used to verify the assumptions of both these lemmas.

\begin{proof}[Proof of Lemma \ref{lemma:subgradient-additivity}]

First, we show  $\neg$\eqref{complete-condition} $\Rightarrow \neg$\eqref{subgradient-additivity}.  By $\neg$\eqref{complete-condition}, there exists $\q \perp \x_0$ such that $\q\q^t \in S$ but
$\y\otimes \q \notin S$ for some $\y$.  Following the calculation in the proof of Theorem \ref{theorem:necessity},
all feasible $\X$ are orthogonal to $\y \otimes \q$.  Hence, $\y \otimes \q \in \partial \indposaxb(\X_0)$, but $\y\otimes \q \notin S = \partial \indpos(\X_0) + \partial \indaxb(\X_0)$.

Next, we show \eqref{complete-condition} $\Rightarrow$ \eqref{subgradient-additivity}.  
One inclusion in \eqref{subgradient-additivity} is automatic:
\begin{align}
\partial \indposaxb (\X_0) = \partial ( \indpos + \indaxb) (\X_0)  \supset \partial \indpos(\X_0) + \partial \indaxb(\X_0).
\end{align}
To prove the other inclusion, we let $\Y \notin S = \partial \indpos(\X_0) + \partial \indaxb(\X_0)$ and show that $\Y \notin \partial \indposaxb(\X_0)$.  By definition of the subgradient, if there exists a feasible $\X$ such that $ {\langle \Y, \X - \X_0\rangle > 0}$, then $\Y \notin \partial \indposaxb(\X_0)$.    We will exhibit such an $\X$ by appealing to the separating hyperplane theorem.  

As we will prove in Lemma \ref{lemma:closed}, \eqref{complete-condition} implies that $S$ is closed.  The separating hyperplane theorem guarantees that we can separate $S$ from any $\Z \notin S$.  That is, for any $\Z \notin S$, there exists a $\lamz$ such that $\langle \lamz, \Z\rangle>0$ and $ \langle \lamz, \M \rangle \leq 0$ $\forall \ \M \in S$.  As a consequence, $\lamz$ satisfies
\begin{align}
\A(\lamz) &= 0, \label{hyperplane-property-first}\\
\langle \lamz, \Q \rangle &\leq 0 \text{ for all } \Q \matrixleq 0, \Q \perp \X_0, \label{hyperplane-property-second}\\
\< \lamz, \M \> &= 0 \text{ if }  \M \in S \text{ and } -\M \in S, \label{hyperplane-property-against-line}\\
\langle \lamz, \Z \rangle &> 0. \label{hyperplane-property-last}
\end{align}
We observe that \eqref{hyperplane-property-second} implies $\proj{\x_0^\perp} \lamz \matrixgeq 0$.
Let $B = \{ \q\q^t \mid \q \perp \x_0, \q\q^t \notin S\}$.  We will build a $\lamtilde$ satisfying the following properties:
\begin{align}
\A(\lamtilde) &= 0, \label{hyperplane-property-first-lamtilde}\\
\langle \lamtilde, \Q \rangle &\leq 0 \text{ for all } \Q \matrixleq 0, \Q \perp \X_0, \label{hyperplane-property-second-lamtilde}\\
\< \lamtilde, \M \> &= 0 \text{ if } \M \in S \text{ and } -\M \in S, \label{hyperplane-property-against-line-lamtilde}\\
\langle \lamtilde, \q\q^t \rangle &> 0 \text{ for all } \q\q^t\in B. \label{hyperplane-property-last-lamtilde}
\end{align}
We build $\lamtilde$ through the following process.  Begin with $\tilde{B} = B$.  
\begin{enumerate}
\item Choose $\q_i \q_i^t \in \tilde{B}$.
\item Let $\lam_{\q_i \q_i^t}$ be given from \eqref{hyperplane-property-first}--\eqref{hyperplane-property-last}.
\item Remove from $\tilde{B}$ any elements that are not orthogonal to $\q_i \q_i^t$.
\item Repeat until $\tilde{B}$ is empty.  
\end{enumerate}
This process terminates in a finite number of repetitions  because the set $\tilde{B}$ is restricted to a space of strictly decreasing dimension at each step.  It produces a finite sequence of $\q_i$ such that all elements of $B$ have positive inner product with $\q_i \q_i^t$ for some $i$.  
  Let $\lamtilde = \sum_i \lam_{\q_i \q_i^t}$.  We observe \eqref{hyperplane-property-first-lamtilde}--\eqref{hyperplane-property-against-line-lamtilde} hold due to \eqref{hyperplane-property-first}--\eqref{hyperplane-property-against-line}.  We have  \eqref{hyperplane-property-last-lamtilde} because     every element of $B$ has a nonnegative inner product with all  $\lam_{\q_i \q_i^t}$ and a strictly positive inner product with at least one $\lam_{\q_i \q_i^t}$.  

We now construct the feasible $\X$ such that $\langle \Y, \X - \X_0\rangle > 0$.  Let $\lam = \lamy + \eps \lamtilde$, where $\eps$ is small enough that $\<\lam, \Y\> >0$. Lemma \ref{lemma:positive-perturbation} guarantees that there exists $\delta > 0$ such that $\X_0 + \delta \lam \matrixgeq 0$.  We take $\X = \X_0 + \delta \lam$. Because $\X \matrixgeq 0$ and $\A(\lam) = 0$, $\X$ is feasible.  Additionally, $\<\Y, \X-\X_0\> > 0$ because $\<\lam, \Y \> >0$.   Hence, $\Y \notin \partial \indposaxb(\X_0)$. 

All that remains is to show that the conditions of Lemma \ref{lemma:positive-perturbation} hold.  The lemma states that if (a) $\proj{\x_0^\perp} \lam \matrixgeq 0$ and (b) $\lam \perp \q \q^t \text{ and } \q \perp \x_0 \Rightarrow \lam \perp \x_0 \otimes \q$, then there exists $\delta > 0$ such that $\X_0 + \delta \lam \matrixgeq 0$.  By \eqref{hyperplane-property-second} and \eqref{hyperplane-property-second-lamtilde}, (a) holds.  To show (b) holds, we consider a $\q\q^t \perp \lam$, $\q \perp \x_0$.    
By  \eqref{hyperplane-property-second} and \eqref{hyperplane-property-second-lamtilde}, we have that $\langle \lamy, \q \q^t \rangle \geq 0$ and $\langle \lamtilde, \q \q^t \rangle \geq 0$.  Hence $\lam \perp \q \q^t \Rightarrow \lamtilde \perp \q \q^t$.  Now \eqref{hyperplane-property-last-lamtilde} implies $\q \q^t \notin B$.  This implies that $\q \q^t \in S$.  By the completeness condition \eqref{complete-condition}, $ \x_0 \otimes \q \in S$ and $-\x_0 \otimes \q \in S$.  Hence, by \eqref{hyperplane-property-against-line} and \eqref{hyperplane-property-against-line-lamtilde}, $\lam \perp \x_0 \otimes \q$, and (b) holds.  

\end{proof}

The hyperplane separation argument above requires that $S$ be closed.  The following lemma reduces the closedness of $S\subset \Rsymn$ to an $n-1\times n-1$ case without the orthogonality constraint, which is proved in Lemma \ref{lemma:closed-no-perp}.

\begin{lemma} \label{lemma:closed}
If  $S = \left \{\sum_i \lambda_i \bfA_i + \Q  \mid \Q \matrixleq 0,  \Q \perp \X_0 \right \}$
satisfies the completeness condition \eqref{complete-condition} then $S$ is closed.
\end{lemma}
\begin{proof}[Proof of Lemma \ref{lemma:closed}]
Without loss of generality let $\X_0 = \e_1 \e_1^t$.  This can be seen by letting $\V$ be an orthogonal matrix with $\x_0/\|\x_0\|$ in the first column, and by considering the set $\V^t S \V$.

If necessary, linearly recombine the $\bfA_i$ such that the first column of $\bfA_1, \ldots, \bfA_\ell$ are independent and the first columns of the remaining $\bfA_{\ell+1}, \ldots, \bfA_m$ are zero.  

Consider a Cauchy sequence ${\Ak + \Qk \rightarrow \X}$, where $\Ak = \sum_{i=1}^m \lik \bfA_i$.  We will establish that $\X \in S$.   Because $\Qk \matrixleq 0$ and $\Qk \perp \e_1 \e_1^t$, $\Qk$ is zero in the first row and column.  Hence the first column of $\sum_{i=1}^\ell \lik \bfA_i$ converges to the first column of $\X$.  By independence, we obtain that $\lik$ converges to some $\liinf$ for each $1 \leq i \leq \ell$.
As a result, $$\sum_{i=\ell+1}^m \lik \bfA_i + \Qk \rightarrow \overline{\X},$$
where $\overline{\X} = \X - \sum_{i=1}^\ell \liinf \bfA_i$, and $\overline{\X}$ is zero in the first row and column.  

The problem has now been reduced to one of size $n-1 \times n-1$ without an orthogonality constraint, and Lemma \ref{lemma:closed-no-perp} can be used to complete the proof.  We now show that the condition of Lemma \ref{lemma:closed-no-perp} holds.  Let $\tilde{\bfA}_i$ be the lower-right $n-1 \times n-1$ submatrix of $\bfA_i$.  Let $\tilde{S} = \{ \sum_{i = \ell + 1}^m \lambda_i \tilde{\bfA}_i + \tilde{\Q} \mid \tilde{\Q} \matrixleq 0 \} \subset \Rsym_{n-1}$.  If $\tilde{\q} \tilde{\q}^t \in \tilde{S}$ then $\begin{pmatrix}0\\\tilde{\q}\end{pmatrix}\begin{pmatrix}0\\\tilde{\q}\end{pmatrix}^t \in S$.  By the completeness condition \eqref{complete-condition}, $\begin{pmatrix}0\\\tilde{\y}\end{pmatrix}\otimes\begin{pmatrix}0\\\tilde{\q}\end{pmatrix} \in S \ \forall \tilde{\y} \in \R^{n-1}$.  By independence of the first columns of $\bfA_1, \ldots, \bfA_\ell$, $\tilde{\y} \otimes \tilde{\q} \in \tilde{S}$.  The conditions of Lemma \ref{lemma:closed-no-perp} are met.  Hence, 
$\overline{\X} = \sum_{i = \ell + 1}^m \lambda_i^{(\infty)} \bfA_i + \Q^{(\infty)} \text{ with } \Q^{(\infty)} \matrixleq 0, \Q^{(\infty)} \perp \e_1 \e_1^t$.  We conclude $\X \in S$.  Thus,  $S$ is closed.  

\end{proof}

The closedness of $S$ above relies on the closedness of a lower dimensional $\tilde{S}$ without the orthogonality constraint.  Closedness is not automatic because it may be that $\sum_i \lik \bfA_i$ and $\Qk$ diverge separately but converge when added together.  We show that this pathology is impossible.  

\begin{lemma} \label{lemma:closed-no-perp}
The set $\tilde{S} = \left \{\sum_i \lambda_i \bfA_i + \Q  \mid \Q \matrixleq 0 \right \} \subset \Rsymn$ is closed if 
\begin{align}\q\q^t \in \tilde{S} \Rightarrow \y\otimes \q \in \tilde{S} \ \forall \y. \label{complete-condition-no-perp}
\end{align}
\end{lemma}

\begin{proof}[Proof of Lemma \ref{lemma:closed-no-perp}]
Consider a Cauchy sequence $\Ak + \Qk \rightarrow \X$, where $\Ak = \sum_i \lik \bfA_i$.  Our goal is to show that $\X \in \tilde{S}$.  Let $V = \Span \{\q \mid \q\q^t \in \tilde{S} \}$.  By \eqref{complete-condition-no-perp}, we note that
\begin{align}
\y \otimes \q \in \tilde{S}  \text{ for all } \y \in \R^n,\ \q \in V.  \label{yq-inclusion}
\end{align}
Recall that that $\projvp \X = \ivp \X \ivp$ is the projector of $\X$ onto matrices with row and column spans in $V^\perp$.  The inclusion \eqref{yq-inclusion} guarantees that for all $\X$,
\begin{align}
\X - \projvp \X &\in \tilde{S}, \label{difference-from-projection-in-span-plus} \\
\projvp \X - \X &\in \tilde{S}. \label{difference-from-projection-in-span-minus}
\end{align}
Taking the projection of the Cauchy sequence, we have
\begin{align}
 \projvp \Ak + \projvp \Qk \rightarrow \projvp \X. \label{proj-cauchy-sequence}
\end{align}
Either $\projvp \Ak$ are bounded or there is an unbounded subsequence.  We will show that the latter is impossible and that the former ensures $\X \in \tilde{S}$.  

We now show the impossibility of $\| \projvp \Ak \|_F \to \infty$ for any subsequence in $k$.  
If $\| \projvp \Ak \|_\text{F} \to \infty$, then $\frac{\| \projvp \Ak \|_\text{F}}{\| \projvp \Qk \|_\text{F}} \to 1$  and 
\begin{align}
\left \langle \frac{\projvp \Ak}{\| \projvp \Ak \|_\text{F}} , \frac{\projvp \Qk}{\|\projvp \Qk\|_\text{F} } \right \rangle \to -1 \text{ as } k \to \infty. \label{sequence-inner-product-minus-1}
\end{align}
The sets $\{\bfA \in \projvp \Span \bfA_i\} \cap \{ \|\bfA \|_\text{F} = 1 \}$ and $\{\Q \matrixleq 0\} \cap \{\|\Q\|_\text{F}=1\}$ are compact.  Hence $\langle \bfA, \Q \rangle$ achieves its minimum value over this set of arguments.  By the Cauchy-Schwarz inequality, the minimum value must be no smaller than $-1$.  As \eqref{sequence-inner-product-minus-1} exhibits a sequence approaching this value, the minimum achieved inner product is $-1$.  That is, there exists $\Q = -\projvp \sum_i \lambda_i \bfA_i$, where $\Q \matrixleq 0$, $\Q = \projvp \Q$, and $\|\Q\|_F=1$.  By applying \eqref{difference-from-projection-in-span-minus} to $\sum_i \lambda_i \bfA_i$, we see $-\Q - \sum_i \lambda_i \bfA_i \in \tilde{S}$.  As $S$ is a convex cone and $\sum_i \lambda_i \bfA_i  \in \tilde{S}$, we observe $-\Q \in \tilde{S}$.  As $-\Q \matrixgeq 0$ is nonzero, it has a positive-semidefinite rank-one component that is also in $\tilde{S}$.   Because  $\range(\Q) \subset V^\perp$, the vector generating this component belongs to $V^\perp$, which contradicts the definition of $V$.  We have thus shown that $\projvp \Ak$ is bounded.

We now show that the boundedness of $\projvp \Ak$ implies $\X \in \tilde{S}$.  By boundedness, there exists a convergent subsequence of $\projvp \Ak$.  As the projection of $\Span\{\bfA_i\}$ is closed, there exist $\liinf$ such that
\begin{align}
\projvp \sum_i \lik \bfA_i \to \projvp \sum_i \liinf \bfA_i. \label{proj-a-limit}
\end{align}
Hence, $\projvp \Qk$ also converges because the projection of the negative-semidefinite cone is closed.  That is, there is a $\Qinf$ such that
\begin{align}
\projvp \Qk \to \projvp \Qinf. \label{proj-q-limit}
\end{align}
Combining \eqref{proj-cauchy-sequence}, \eqref{proj-a-limit}, and \eqref{proj-q-limit}, we observe that 
\begin{align}
 \projvp \left( \X - \sum_i \liinf \bfA_i - \Qinf \right) = 0 \label{projvp-xaq-zero}
\end{align}
Using \eqref{projvp-xaq-zero} and applying \eqref{difference-from-projection-in-span-plus} to $\X - \sum_i \liinf \bfA_i - \Qinf $, we get that
\begin{align}
\X - \sum_i \liinf \bfA_i - \Qinf  \in \tilde{S}
\end{align}
As $S$ is a convex cone and $\sum_i \liinf \bfA_i + \Qinf \in \tilde{S}$, we conclude $\X \in \tilde{S}$.  

\end{proof}

The following lemma establishes a necessary and sufficient condition for when a symmetric perturbation from a positive-semidefinite rank-one matrix remains positive-semidefinite.    

\begin{lemma} \label{lemma:positive-perturbation}
Let $\X_0 = \x_0 \x_0^t \in \Rsymn$.  $\X_0 + \delta \lam \matrixgeq 0$ for some $\delta > 0$ if and only if  (a) $\projxop \lam \matrixgeq 0$ and (b) $ \lam \perp \q\q^t$ and $\q \perp \x_0 \Rightarrow \lam \perp \x_0 \otimes \q.
$

\end{lemma}
\begin{proof}
Without loss of generality, let $\x_0 = \e_1$ and $\X_0 = \e_1 \e_1^t$.  In this case $\projxop$ is the restriction to the lower-right $n-1\times n-1$ block.  Let $\lamxop \in \Rsym_{n-1}$ be that lower-right block of $\lam$.   Write the block form 
 $$
 \lam = \begin{pmatrix} \Lambda_{11} & \brhot^t \\ \brhot & \lamxop \end{pmatrix}.
 $$

First, we prove $\X_0 + \delta \lam \matrixgeq 0 \Rightarrow $ (a) and (b).  We immediately have (a) because $\X_0$ is zero on the lower-right subblock.  
To establish (b), we first consider the case where $1 + \delta \Lambda_{11} = 0$.  By $\X_0 + \delta \lam \matrixgeq 0$, we have that $\brhot=0$ and (b) holds.  Now, we consider the case that $1 + \delta \Lambda_{11} > 0$.  Further consider a $\q$ such that $\lam \perp \q\q^t$ and $\q \perp \e_1$.  Hence, $\q$ can be written as $\begin{pmatrix}0 \\ \tilde{\q}\end{pmatrix}$.  We observe that $\lamxop \perp \tilde{\q}\tilde{\q}^t$.  Using a Schur complement, 
\begin{align}
\text{if } 1+\delta \Lambda_{11} > 0, \text{ then  } \X_0 + \delta \lam \matrixgeq 0 \Leftrightarrow \lamxop - \frac{\delta}{1 + \delta \Lambda_{11}} \brhot \brhot^t \matrixgeq 0.   \label{schur-complement}
\end{align} By \eqref{schur-complement}, we obtain $\brhot^t \tilde{\q}=0$, which implies that $\lam \perp \e_1 \otimes \q$.  

Second, we prove (a) and (b) $\Rightarrow$ $\X_0 + \delta \lam \matrixgeq 0$ for some $\delta>0$.
 Assume (a) and (b) hold.  
Using the property \eqref{schur-complement} about Schur complements, it suffices to show
\begin{align}
 1 + \delta \Lambda_{11} > 0 \text{ and } \lamxop - \frac{\delta}{1 + \delta \lam_{11}} \brhot \brhot^t \matrixgeq 0. \label{schur-complement-condition}
\end{align} 
Let $V\subset \R^{n-1}$ be the range of $\lamxop$.  Taking $\eps$ to be  the smallest nonzero  eigenvalue of $\lamxop$, we note that $\lamxop \matrixgeq \eps \I_{V}$.  
We note that for any $\tilde{\q} \in V^{\perp}$, (b) guarantees $\brhot \perp \tilde{\q}$.  Hence $\brhot \in V$, and there is a sufficiently small $\delta$ such that $\frac{\delta}{1 + \delta \lam_{11}} \brhot \brhot^t \matrixleq  \eps \I_{V}$.  We conclude that \eqref{schur-complement-condition} holds, and hence $\exists \delta >0$ such that $\X_0 + \delta \lam \matrixgeq 0$. 
\end{proof}

\section*{Acknowledgements} P.H. is partially supported by an NSF Mathematical Sciences Postdoctoral Research Fellowship and thanks Laurent Demanet, Radu Balan, Henry Wolkowicz and Yuen-Lam Cheung for useful discussions.

\end{document}